\documentclass[11pt]{article}
\usepackage{amsmath,amsfonts,amsthm,amssymb, mathtools}
\usepackage{mathrsfs, graphicx,color,latexsym, tikz, calc,
}
\usepackage{tikz-cd}
\usepackage{graphicx}
\usepackage{epstopdf}
\usepackage{enumerate}
\usepackage{caption}
\usepackage{stmaryrd}
\usepackage{hyperref}

\usetikzlibrary{shadows}
\usetikzlibrary{patterns,arrows,decorations.pathreplacing}
\usepackage{ulem}
\newtheorem{theorem}{Theorem}[section]
\newtheorem{lemma}[theorem]{Lemma}
\newtheorem{definition}[theorem]{Definition}
\newtheorem{problem}[theorem]{Problem}
\newtheorem{conjecture}[theorem]{Conjecture}

\newtheorem{claim}{Claim}
\newtheorem*{claim*}{Claim}

\newtheorem{example}[theorem]{Example}

\allowdisplaybreaks[4]

\title{\bf \Large }
\date{ }
\title{\bf \Large 
	Two results on set families: Sturdiness and intersection\footnote{This paper was published on Journal of Combinatorial Theory, Series A 220 (2026) 106153. E-mail addresses: \url{fenglh@163.com} (L. Feng, corresponding author),
    \url{ytli0921@hnu.edu.cn} (Y. Li, corresponding author).
    }}
\author{
{\small  Yongjiang Wu$^1$, \ \  Zhiyi Liu$^1$,  \ \ Lihua Feng$^1$, \ \ Yongtao Li$^2$}\\[2mm]
\small $^1$School of Mathematics and Statistics, HNP-LAMA, Central South University\\
 \small Changsha, Hunan, 410083, China\\ 
 \small $^2$Yau Mathematical Sciences Center, Tsinghua University, Beijing, 100084, China\\
}

\begin{document}
\maketitle
\begin{abstract}
This paper resolves two open problems in extremal set theory. For a family $\mathcal{F} \subseteq 2^{[n]}$ and $i, j\in [n]$, we denote $\mathcal{F} (i,\bar{j})=\{F\backslash\{i\}:  F\in \mathcal{F}, F\cap\{i,j\}=\{i\}\}$. The sturdiness $\beta (\mathcal{F})$ is defined as the minimum $|\mathcal{F} (i,\bar{j})|$ over all $i\neq j$. A family $\mathcal{F}$ is called an IU-family if it satisfies the intersection constraint: $F\cap F'\neq \emptyset $ for all $F,F'\in \mathcal{F}$, as well as the union constraint: $F\cup F' \neq [n]$ for all $F,F'\in \mathcal{F}$. The well-known IU-Theorem states that every IU-family $\mathcal{F}\subseteq 2^{[n]}$ has size at most $ 2^{n-2}$. 
In this paper, we prove that 
if $\mathcal{F}\subseteq 2^{[n]}$ is an IU-family, then $\beta (\mathcal{F})\le 2^{n-4}$. 
This confirms a recent conjecture proposed by Frankl and Wang. 

As the second result, we establish a tight upper bound on the sum of sizes of cross $t$-intersecting separated families. Our result not only extends a previous theorem  of Frankl, Liu, Wang and Yang on separated families, 
but also provides explicit counterexamples to an open problem proposed by them, thereby settling their problem in the negative.\\

\end{abstract}

\section{Introduction}
This paper addresses two open problems in extremal set theory concerning the structural stability and intersection properties of set systems. We begin by establishing the necessary notation and foundational concepts that will be used throughout our analysis.

 For integers $m\leq n$, let  $[m, n]=\{m, m+1, \ldots, n\}$ and simply let $[n]=[1,n]$. Let $2^{[n]}$ denote the family of all subsets of $[n]$, and let  $\binom{[n]}{k}$ denote the family of all $k$-element subsets of $[n]$.  
 We denote $\binom{[n]}{\leq k}=\{F\subseteq [n]: |F|\leq k\}$. A family $\mathcal{F}\subseteq 2^{[n]}$  is called $k$-\textit{uniform} if $\mathcal{F}\subseteq\binom{[n]}{k}$, otherwise it is \textit{non-uniform}. 
 A family $\mathcal{F} \subseteq 2^{[n]}$ is called $t$-\textit{intersecting} if $|F\cap F^{\prime}|\geq  t$
 for all $F, F^{\prime}\in \mathcal{F}$. 
Two families $ \mathcal{F}, \mathcal{G} \subseteq 2^{[n]} $ are called  \textit{cross $t$-intersecting} if $|F\cap G|\geq  t$
 for all $F\in \mathcal{F}$ and $G\in \mathcal{G}$.
For the case $t=1$, we abbreviate $1$-intersecting and cross $1$-intersecting as intersecting and cross intersecting, respectively.
A family $\mathcal{F} \subseteq 2^{[n]}$ is called $s$-\textit{union} if $|F\cup F^{\prime}|\leq  s$
 for all $F, F^{\prime}\in \mathcal{F}$. 
For a family $\mathcal{F}\subseteq 2^{[n]}$ and $i\in [n]$, we denote 
$\mathcal{F} (i)=\{F\backslash\{i\}: i\in F\in \mathcal{F}\}$ and  $\mathcal{F} (\bar{i})=\{F: i\notin F\in \mathcal{F}\}$.

In extremal set theory, the structural analysis of set families often relies on two key parameters: diversity and sturdiness. For a family $\mathcal{F} \subseteq 2^{[n]}$, its \textit{diversity} is defined as 
\[ \gamma(\mathcal{F}) = \min_{i \in [n]} |\mathcal{F}(\bar{i})|, \] 
which measures the minimum number of sets avoiding any given element. Equivalently, we see that $\gamma (\mathcal{F})=|\mathcal{F}|-\Delta(\mathcal{F})$, where $ \Delta(\mathcal{F})=\max_{i \in [n]} |\mathcal{F}(i)|$ denotes the \textit{maximum degree} of 
$\mathcal{F}$. 
Diversity has become a very important tool for the study of intersecting families. This parameter helps to distinguish whether a family is close to a star (where diversity equals zero). 
Answering a question of Katona who asks: what is the maximum possible diversity of an $k$-unform intersecting family $\mathcal{F}$ of sets of $[n]$? Lemons and Palmer \cite{LP2008} 
showed that 
\[ \gamma (\mathcal{F}) \le {n-3 \choose k-2} \] 
for every $n\ge 6k^3$. Subsequently, 
this bound on $n$  was improved to $n\ge 6k^2$ by Frankl \cite{F17}, to $n\ge 72k$ by Frankl \cite{F20}, and to $n\ge 36k$ by Frankl and Wang \cite{FW2024}. The best possible bound on $n$ remains unknown for this problem. 
On the other hand, some examples provided by Huang \cite{Huang2019} and Kupavskii \cite{Kup2018} showed that for $n< (2+ \sqrt{3})k$, there exists a $k$-uniform  intersecting family $\mathcal{F}$ with diversity $\gamma (\mathcal{F}) > {n-3 \choose k-2}$. 
For non-uniform families, Frankl and Kupavskii \cite{FK2021} proved that the largest diversity of an intersecting family $\mathcal{F}\subseteq 2^{[n]}$ satisfies $\gamma (\mathcal{F}) = \big(1- \Theta (\frac{\log n}{n}) \big) 2^{n-2}$.  
We refer to \cite{FK2021} for several results for $t$-intersecting families and \cite{FW2022, FW2025, Pat2025} for related variants.

\subsection{The sturdiness of IU-families}

Building on the concept of diversity, Frankl and Wang \cite{FW-stur} recently introduced the notion of sturdiness. To start with, 
we denote $ \mathcal{F} (i,\bar{j})=\{F\backslash\{i\}:  F\in \mathcal{F}, F\cap\{i,j\}=\{i\}\}$. 
The \textit{sturdiness} $\beta(\mathcal{F})$ is  defined as 
$$\beta(\mathcal{F}) = \min_{1 \leq i \neq j \leq n} |\mathcal{F}(i,\bar{j})|.$$ 
 This parameter provides a more refined measure of a family's local robustness. Using diversity as one of their main tools, Frankl and Wang \cite{FW-stur} proved that if $n \geq 36(k+6)$ and $\mathcal{F} \subseteq \binom{[n]}{k}$ is an intersecting family, then 
 \[ \beta(\mathcal{F}) \leq \binom{n-4}{k-3}. \] 
 Beyond intersecting families, Frankl and Wang also investigated the maximum sturdiness for $t$-intersecting families, exploring how intersection constraints impact local robustness.
More recently,  Patk\'{o}s \cite{Pat2025} investigated a generalization of sturdiness through the concept of $(p,q)$-d\H{o}m\H{o}d\H{o}m, which encompasses several parameters including diversity, minimum degree and sturdiness as special cases.

 In  \cite{FW-stur},
 Frankl and Wang proposed several conjectures about the maximum sturdiness, including an interesting conjecture for the \textit{IU-family} $\mathcal{F}$, which satisfies 
 \begin{equation} \label{eq-cap}
     F\cap F' \neq \emptyset ~~\text{for all $F,F'\in \mathcal{F}$}, 
 \end{equation}
 and 
 \begin{equation}
     \label{eq-cup}
     F\cup F' \neq [n] ~~\text{for all $F,F'\in \mathcal{F}$}.
 \end{equation}
 We have so far encountered problems related to both intersections and unions. As noted previously, an intersection problem is essentially equivalent to the corresponding union problem when applied to the family of complements. However, when both intersection and union conditions are imposed on the same family, new and distinct types of problems arise. Let us introduce the first non-trivial case of this kind. 

 \begin{theorem}[IU-Theorem]
 If $\mathcal{G}\subseteq 2^{[n]}$ is an IU-family, i.e., satisfying (\ref{eq-cap}) and (\ref{eq-cup}), then
\[ |\mathcal{F}| \le 2^{n-2}.  \] 
\end{theorem}

 This result was independently studied by several authors, including Kleitman \cite{Kle1966}, 
 Marica and Sch\"{o}nheim \cite{MS1969}, 
 Daykin and Lov\'{a}sz \cite{DL1975}, 
 Seymour \cite{Sey1973}, 
 Anderson \cite{And1976}, 
 and Frankl \cite{Fra1975}. 
 We refer to \cite[Theorem 6.3]{FT2016} for an elegant proof by using the Kleitman correlation inequality. 

 While the maximum size of IU-families is well understood, recent work has focused on finer structural properties such as their local robustness. This leads to the following conjecture \cite{FW-stur} regarding the sturdiness of IU-families.
 
\begin{conjecture}[Frankl and Wang \cite{FW-stur}] \label{cj1} 
Let $\mathcal{G}\subseteq 2^{[n]}$ be an IU-family. Then
$$
\beta(\mathcal{G})\leq 2^{n-4}.
$$
\end{conjecture}

%This conjecture reveals a deep connection between the global structure of IU-families and their local robustness properties.

For $\mathcal{G}\subseteq 2^{[n]}$ and $X\subseteq [n]$, let $\mathcal{G}_{| X}=\{G\cap X: G\in \mathcal{G}\}$.
Frankl and Wang \cite{FW-stur} made partial progress by showing that if there exists a partition $[n]=X\cup Y$ such that $\mathcal{G}_{| X}$ is intersecting and $\mathcal{G}_{| Y}$ is $(|Y|-1)$-union, then indeed $\beta(\mathcal{G})\leq 2^{n-4}$.

Our first result of this paper confirms Conjecture \ref{cj1}.

\begin{theorem}\label{t1}
For any IU-family $\mathcal{G}\subseteq 2^{[n]}$, we have 
$
\beta(\mathcal{G})\leq2^{n-4}.
$
\end{theorem}

\subsection{Upper bound for separated families}

 Let $m,k$ be positive integers and suppose $[m]=\bigcup_{i\in[k]}X_i$,  where $X_i, i\in [k]$ are pairwise disjoint and $|X_i|=n$. 
Assume further that the elements of $X_i$ are ordered.
Denote the smallest element of $X_i$ by $v_i$.
For any $\ell\in [k]$, we define 
$$\mathcal{H}(n,k,\ell)=\left\{H\in\binom {[m]}{\ell}:|H\cap X_i|\le 1,i\in [k]\right\}.$$ 
A family $\mathcal{F}\subseteq \mathcal{H}(n,k,\ell)$ is called  a \textit{separated family}. Two separated families  $\mathcal{F}\subseteq\mathcal{H}(n,k,\ell)$ and $\mathcal{G}\subseteq\mathcal{H}(n,k,\ell')$ are said to be \textit{cross $t$-intersecting} if $|F\cap G|\geq  t$
 for all $F\in \mathcal{F}$ and $G\in \mathcal{G}$. For $t=1$, we simply
say that they are cross intersecting.

Recently, Frankl, Liu, Wang and Yang \cite{F23} determined the maximum possible sum of sizes of non-empty cross intersecting separated families.

\begin{theorem}\cite{F23}\label{F23}
 Let $n,k,\ell,\ell'$ be integers with $n>3\ell$ and $k\ge \ell\ge \ell' \ge 2$. We define 
\begin{align*}
\mathcal{F}_0'=&\left \{ F\in\mathcal{H}(n,k,\ell) : F\cap\{v_1,\ldots,v_{\ell'}\}\neq \emptyset\right \} ~\text{and}~ \mathcal{G}_0'=\left \{ \{v_1,\ldots,v_{\ell'}\} \right \}.
     \end{align*}
Suppose that $\mathcal{F}\subseteq\mathcal{H}(n,k,\ell)$ and $\mathcal{G}\subseteq\mathcal{H}(n,k,\ell')$ are both non-empty  and cross intersecting, then
\begin{align*}
        |\mathcal{F}|+ |\mathcal{G}|\le |\mathcal{F}_0'|+|\mathcal{G}_0'|= \sum\limits_{j\in[\ell]} 
     \left(\binom{k}{j}-\binom{k-\ell'}{j}\right)\binom{k-j}{\ell-j}(n-1)^{\ell-j}+1.
     \end{align*}
\end{theorem}

In the same paper, the authors  proposed the following open problem.

\begin{problem}\cite{F23}\label{p1}
Does the conclusion of Theorem \ref{F23} hold for all values of $n\geq 2$, irrespective the values of $k\ge \ell\ge \ell'\geq 1$?
\end{problem}

To address this problem, we consider a more general framework of cross $t$-intersecting separated families. Within this setting, we establish the following sharp  bound on the sum of sizes of cross intersecting separated families.

\begin{theorem}\label{t3}
 Let $n,k,\ell,\ell'$ be integers with $n\geq 2$ and $k\ge \ell\ge \ell' \ge t\geq 1$. We define 
\begin{align*}
\mathcal{F}_0=\left \{ F\in\mathcal{H}(n,k,\ell) :|F\cap\{v_1,\ldots,v_{\ell'}\}|\ge t\right \} ~\text{and}~  \mathcal{G}_0=\left \{\{v_1,\ldots,v_{\ell'}\} \right \}.
     \end{align*}
 For every $a\ge 1$, we define  
\[ \mathcal{F}_a=\left \{ F\in\mathcal{H}(n,k,\ell) :\{v_1,\ldots,v_a\}\subseteq F\right \} ~\text{and}~ 
\mathcal{G}_a=\left \{ G\in\mathcal{H}(n,k,\ell') :|G\cap\{v_1,\ldots,v_a\}|\ge t \right \}.\]
If $\mathcal{F}\subseteq\mathcal{H}(n,k,\ell)$ and $\mathcal{G}\subseteq\mathcal{H}(n,k,\ell')$ are both non-empty  and cross $t$-intersecting, then
\begin{align*}
        |\mathcal{F}|+ |\mathcal{G}|\le\max \left\{|\mathcal{F}_0|+|\mathcal{G}_0|,\ \max_{a\in [\ell'+1 ,\ell] }\left\{|\mathcal{F}_a|+|\mathcal{G}_a|\right\}\right\}.
     \end{align*}
\end{theorem}

This general result not only extends Theorem \ref{F23} ($t=1$) but also generalizes a main result of Liu \cite{L23} (in the case $\ell=\ell'$, the bound reduces to $|\mathcal{F}_0| + |\mathcal{G}_0|$).
By constructing explicit counterexamples guided by Theorem \ref{t3}, we show that the conclusion of Theorem \ref{F23} does not hold in general for all \( n \geq 2 \) and all \( k \ge \ell \ge \ell' \geq 1 \). We thereby provide a negative answer to Problem \ref{p1}.

\begin{example}\label{ex1}
Under the notation of Theorem \ref{t3}.
Let $t=1$ and $(n, k, \ell, \ell')=(2,3,3,2)$. 
Let $\mathcal{F}\subseteq\mathcal{H}(2,3,3)$ and $\mathcal{G}\subseteq\mathcal{H}(2,3,2)$ be non-empty cross intersecting families.
By Theorem \ref{t3}, we have
$$
 |\mathcal{F}|+ |\mathcal{G}|\le\max \left\{|\mathcal{F}_0|+|\mathcal{G}_0|,\ |\mathcal{F}_3|+|\mathcal{G}_3|\right\}.
$$
By simple calculations, we get
\begin{align*}
|\mathcal{F}_0|+|\mathcal{G}_0| =\sum\limits_{j\in[3]} 
     \left(\binom{3}{j}-\binom{3-2}{j}\right)\binom{3-j}{3-j}(2-1)^{3-j}+1=7
     \end{align*}
     and 
     \[ |\mathcal{F}_3|+|\mathcal{G}_3|=\sum\limits_{j\in[2]} 
     \left(\binom{3}{j}-\binom{3-3}{j}\right)\binom{3-j}{2-j}(2-1)^{2-j}+1=10. \]
However, $\mathcal{F}_3$ and $\mathcal{G}_3$ are cross intersecting and $|\mathcal{F}_3|+|\mathcal{G}_3|>|\mathcal{F}_0|+|\mathcal{G}_0|$.
Moreover, let $t=1$ and  $(n, k, \ell, \ell')=(2,4,4,2)$. We similarly have 
$|\mathcal{F}_4|+|\mathcal{G}_4|>|\mathcal{F}_0|+|\mathcal{G}_0|$.

\end{example}

\noindent 
{\bf Organization.} 
The remainder of this paper is structured as follows. In Section \ref{se2}, we prove Theorem \ref{t1}. In Section \ref{se3}, we prove Theorem \ref{t3}.

\section{Proof of Theorem \ref{t1}} \label{se2}

Our proof relies on a foundational result of Seymour \cite{Sey1973} concerning cross-Sperner families. Recall that two families $\mathcal{F},\mathcal{G}\subseteq 2^{[n]}$ are called \textit{cross-Sperner} if for all $F\in\mathcal{F}$ and $G\in\mathcal{G}$, we have neither $F\subseteq G$ nor $G\subseteq F$. To make the presentation self-contained and to serve our purposes more effectively, we have derived a strengthened form of this result (Theorem \ref{t2}). We include its proof, which is inspired by Seymour's original arguments.

To facilitate our analysis, we recall some  notions. A family $\mathcal{F} \subseteq 2^{[n]}$ is called a \textit{downset} (resp. an \textit{upset}) if for every $A \in \mathcal{F}$ and every $B \subseteq [n]$ satisfying $B \subseteq A$ (resp. $A \subseteq B$), it holds that $B \in \mathcal{F}$.
For any $p\in (0,1)$ and $F\in \mathcal{F} \subseteq 2^{[n]}$, the \textit{product measure} $\mu_p$ on $F$ is defined by $\mu_p(F)=p^{|F|}(1-p)^{n-|F|}$. We extend this measure to the whole family by setting $\mu_p(\mathcal{F}) = \sum_{F \in \mathcal{F}} \mu_p(F)$.

A key ingredient in our derivation is the following correlation inequality:

\begin{lemma}\cite{S12}\label{S12}
Let  $\mathcal{F} \subseteq 2^{[n]}$ be a downset and $\mathcal{G} \subseteq 2^{[n]}$ be an upset. Then
$$
\mu_p(\mathcal{F}\cap \mathcal{G})\leq\mu_p(\mathcal{F})\mu_p(\mathcal{G}).
$$
\end{lemma}

\begin{theorem}\label{t2}
Let $\mathcal{F}, \mathcal{G} \subseteq 2^{[n]}$ be cross-Sperner families. Then 
$$
\mu_p(\mathcal{F})^{\frac{1}{2}}+\mu_p(\mathcal{G})^{\frac{1}{2}}\leq 1.
$$
\end{theorem}
\begin{proof}
Let 
\begin{align*}
\mathcal{A}=&\left\{A\subseteq [n]: \exists~ F\in \mathcal{F} \text{ s.t. } F\subseteq A \text{ and } \exists~ G\in \mathcal{G} \text{ s.t. } G\subseteq A\right\},\\
\mathcal{B}=&\left\{B\subseteq [n]: \exists~ F\in \mathcal{F} \text{ s.t. } F\subseteq B \text{ but } \nexists~ G\in \mathcal{G} \text{ s.t. } G\subseteq B \right\},\\
\mathcal{C}=&\left\{C\subseteq [n]: \nexists~ F\in \mathcal{F} \text{ s.t. } F\subseteq C \text{ but } \exists~ G\in \mathcal{G} \text{ s.t. } G\subseteq C \right\},\\
\mathcal{D}=&\left\{D\subseteq [n]: \nexists~ F\in \mathcal{F} \text{ s.t. } F\subseteq D \text{ and } \nexists~ G\in \mathcal{G} \text{ s.t. } G\subseteq D \right\}.
\end{align*}
Observe that $\mathcal{A}\cup \mathcal{B}$ is an upset and $\mathcal{B}\cup \mathcal{D}$ is a downset.
By Lemma \ref{S12}, we have 
$$
\mu_p((\mathcal{A}\cup \mathcal{B})\cap (\mathcal{B}\cup \mathcal{D}))\leq\mu_p(\mathcal{A}\cup \mathcal{B})\mu_p(\mathcal{B}\cup \mathcal{D}).
$$
Since $\mathcal{A}, \mathcal{B},\mathcal{C},\mathcal{D}$
are pairwise disjoint, we obtain
$$
\mu_p(\mathcal{B})\leq\left(\mu_p(\mathcal{A})+\mu_p(\mathcal{B})\right)\left(\mu_p(\mathcal{B})+\mu_p(\mathcal{D})\right).
$$
Since $2^{[n]}=\mathcal{A}\cup \mathcal{B}\cup \mathcal{C}\cup \mathcal{D}$ implies $\mu_p(\mathcal{A})+\mu_p(\mathcal{B})+\mu_p(\mathcal{C})+\mu_p(\mathcal{D})=1$, we further have
$$
\mu_p(\mathcal{B})\left( \mu_p(\mathcal{A})+\mu_p(\mathcal{B})+\mu_p(\mathcal{C})+\mu_p(\mathcal{D})\right)\leq\left(\mu_p(\mathcal{A})+\mu_p(\mathcal{B})\right)\left(\mu_p(\mathcal{B})+\mu_p(\mathcal{D})\right).
$$
It follows that
$$
\mu_p(\mathcal{B})\mu_p(\mathcal{C})\leq\mu_p(\mathcal{A})\mu_p(\mathcal{D}).
$$
Applying the AM-GM inequality yields 
$$
\mu_p(\mathcal{B})\mu_p(\mathcal{C})\leq\mu_p(\mathcal{A})\mu_p(\mathcal{D})\leq \left(\frac{\mu_p(\mathcal{A})+\mu_p(\mathcal{D})}{2}\right)^2=\left(\frac{1-\mu_p(\mathcal{B})-\mu_p(\mathcal{C})}{2}\right)^2.
$$
Consequently, 
$$
\mu_p(\mathcal{B})+\mu_p(\mathcal{C})+2\mu_p(\mathcal{B})^{\frac{1}{2}}\mu_p(\mathcal{C})^{\frac{1}{2}}\leq 1.
$$
Therefore, we have
$$
\mu_p(\mathcal{B})^{\frac{1}{2}}+\mu_p(\mathcal{C})^{\frac{1}{2}}\leq 1.
$$

Since  $\mathcal{F}$ and $\mathcal{G}$ are cross-Sperner, we infer that $\mathcal{F}\subseteq \mathcal{B}$ and $\mathcal{G}\subseteq \mathcal{C}$.
This implies that
$$
\mu_p(\mathcal{F})^{\frac{1}{2}}+\mu_p(\mathcal{G})^{\frac{1}{2}}\leq \mu_p(\mathcal{B})^{\frac{1}{2}}+\mu_p(\mathcal{C})^{\frac{1}{2}}\leq1,
$$
as desired.
\end{proof}

Having established Theorem \ref{t2}, we now proceed to apply this to resolve Conjecture \ref{cj1}.

\begin{proof}[Proof of Theorem \ref{t1}]
Let $x\neq y \in [n] $ and $\mathcal{A}= \mathcal{G}(x,\bar{y})$ and $\mathcal{B}= \mathcal{G}(y,\bar{x})$
Since $\mathcal{G}\subseteq 2^{[n]}$ is an IU-family, it follows that $\mathcal{A}, \mathcal{B} \subseteq 2^{[n]\backslash \{x,y\}}$ are cross intersecting and $(n-3)$-union.

Define $\mathcal{B}^c=\{B^c: B\in \mathcal{B}\}$, where $B^c=([n]\backslash\{x,y\}) \backslash G$. Since $\mathcal{A}$ and $ \mathcal{B} $ are cross intersecting and $(n-3)$-union, it follows that for any $A\in \mathcal{A}$ and  $B\in \mathcal{B}$, we have $A\cap B\neq \emptyset$ and $A\cup B\neq [n]\backslash\{x,y\}$. The former implies $A\nsubseteq B^c$, where the latter implies $B^c\nsubseteq A$. Therefore, $\mathcal{A}$ and $\mathcal{B}^c$ are cross-Sperner. Setting $p=\frac{1}{2}$ in Theorem \ref{t2}, we obtain $
|\mathcal{A}|^{\frac{1}{2}}+|\mathcal{B}^c|^{\frac{1}{2}}\leq 2^{\frac{n-2}{2}}
$. By $|\mathcal{B}|=|\mathcal{B}^c|$ and the AM-GM inequality,  we have
$$
|\mathcal{A}|^{\frac{1}{2}}|\mathcal{B}|^{\frac{1}{2}}\leq\left(\frac{|\mathcal{A}|^{\frac{1}{2}}+|\mathcal{B}|^{\frac{1}{2}}}{2}\right)^2\leq \frac{2^{n-2}}{4}=2^{n-4},
$$
which implies that $\min\{
|\mathcal{A}|,|\mathcal{B}|\}\leq 2^{n-4}.
$
Consequently,
$
\beta(\mathcal{G})\leq\min\{
|\mathcal{A}|,|\mathcal{B}|\}\leq 2^{n-4}.
$
\end{proof}

\section{Proof of Theorem \ref{t3}}\label{se3}
\subsection{Shifting operator and reduction lemma}
We commence our analysis  with the shifting technique \cite{E61}, which is one of the most powerful tools in extremal set theory.
Let  $\mathcal{F} \subseteq 2^{[k]}$ be a family and let $1 \leq i<j \leq k$ be integers. The \textit{shifting operator} $s_{i, j}$ on $\mathcal{F}$ is defined as follows:
$$s_{i, j}(\mathcal{F})=\left\{s_{i, j}(F): F \in \mathcal{F}\right\},$$
where
$$
s_{i, j}(F)= \begin{cases}
(F \backslash\{j\}) \cup\{i\} & \text { if } j \in F, i \notin F \text { and } (F \backslash\{j\}) \cup\{i\} \notin \mathcal{F}; \\ F & \text { otherwise. }\end{cases}
$$
A family $\mathcal{F} \subseteq 2^{[n]}$ is called \textit{shifted} or \textit{left-compressed} if $s_{i, j}(\mathcal{F})=  \mathcal{F}$ holds for all $1 \leq i<j \leq k$.  
It is a classical result that by applying the shifting operator repeatedly, every family will terminate at a shifted family. Notably, the shifting operator preserves the cross $t$-intersection property, as captured by the following lemma.

\begin{lemma}\cite{G23} \label{G23}
 Let $\mathcal{F},\mathcal{G}  \subseteq 2^{[k]}$  be  cross $t$-intersecting families and let $1 \leq i<j \leq k$ be integers.
Then $s_{i, j}(\mathcal{F})$ and $s_{i, j}(\mathcal{G})$ are cross $t$-intersecting.
\end{lemma}

We now extend these ideas to separated family. Consider a ground set partitioned as $[m] = \bigcup_{i \in [k]} X_i$, where the $X_i$ are pairwise disjoint sets of size $n$, and let $v_i$ denote the smallest element in $X_i$. For every $\ell\in [k]$, $$\mathcal{H}(n,k,\ell)=\left\{H\in\binom {[m]}{\ell}:|H\cap X_i|\le 1,i\in [k]\right\}.$$ 
Let $\mathcal{F}\subseteq\mathcal{H}(n,k,\ell)$ be a separated family. The shifting operator $s_{ij}$ is allowed to apply on  $\mathcal{F}$ if $i<j$ are in the same $X_a$ for some $a\in [k]$.
A separated family $\mathcal{F}\subseteq\mathcal{H}(n,k,\ell)$  is called \textit{shifted} if $s_{ij}(\mathcal{F})=\mathcal{F}$ for all $i<j$ in the same $X_a$.
By applying the allowed shifting operator repeatedly, every 
separated family eventually becomes shifted.

For any $H\in \mathcal{H}(n,k,\ell)$ and $\mathcal{F}\subseteq \mathcal{H}(n,k,\ell)$,  we fix the following notation:
$$
A(H)=\left\{i:H\cap X_i=\{v_i\}\right\},\quad \mathcal{A}(\mathcal{F})=\left\{A(F):F\in\mathcal{F}\right\}.
$$
This notation allows us to associate each set with a subset of indices indicating in which blocks it contains exactly the minimal element. We now state a structural result that facilitates our analysis, which is inspired by the reduction lemma in \cite{F23, FLWY2024, L23}.

\begin{lemma}[Reduction lemma]\label{4}
Suppose that $\mathcal{F}\subseteq\mathcal{H}(n,k,\ell)$  and  $\mathcal{G}\subseteq\mathcal{H}(n,k,\ell')$  are shifted and  cross $t$-intersecting. Then $\mathcal{A}(\mathcal{F})\subseteq\binom{[k]}{\le \ell}$ and $\mathcal{A}(\mathcal{G})\subseteq \binom{[k]}{\le \ell'}$ are cross $t$-intersecting.
\end{lemma}
\begin{proof}
    Assume that there exist $F=\left\{x_{a_{1}},\ldots,x_{a_{\ell}} \right\}\in\mathcal{F}$ and $G= \left\{y_{b_{1}},\ldots,y_{b_{\ell'}}\right\}\in\mathcal{G}$ such that $\left|A(F)\cap A(G)\right|<t$, where $x_{a_{j}}\in X_{a_{j}}$  for $j\in[\ell]$ and $y_{b_{j}}\in X_{b_{j}}$  for $j\in[\ell']$. 
     Recall that $v_{a_{i}}$ is the smallest element of $X_{a_{i}}$.
    Let $F'=\left\{z_{a_{1}},\ldots,z_{a_{\ell}}\right\}$, where for $i\in[\ell]$, 
$$
z_{a_{i}}=
\begin{cases}
    x_{a_{i}} &\text{if} \ \ a_{i}\in A(G),\\
    v_{a_{i}} &\text{if} \ \ a_{i}\notin A(G).
\end{cases}
$$
Observe that $F'$ can be obtained from $F$ by allowed shifting operations. Since $\mathcal{F}$ is shifted, we have $F'\in\mathcal{F}$. It follows that
$$
\left|F'\cap G \right|= \left|A(F)\cap A(G) \right|<t,
$$
a contradiction.
\end{proof}

\subsection{Necessary intersection point}

We prove our main auxiliary result (Theorem \ref{n1}) in this part. Its proof utilizes the necessary intersection point method, a powerful technique recently developed by Gupta , Mogge, Piga and Schulke \cite{G23} for analyzing cross intersecting families. 
To apply this method effectively, we first introduce some necessary notions and preliminary results.

\begin{definition}
    Let $\mathcal{F}, \mathcal{G} \subseteq 2^{[k]}$ be cross $t$-intersecting families.  We say $a \in[n]$ is a \textit{necessary intersection point} of $\mathcal{F}$ and $ \mathcal{G}$ if there exist  $F \in \mathcal{F}$ and $ G  \in \mathcal{G} $ such that
    \begin{align*}
    \left|[a] \cap F\cap G\right|=t \text { and } a \in F\cap G.
    \end{align*}
\end{definition}

The above definition naturally leads to the following result.

 \begin{lemma}\label{34}
      Let $\mathcal{F}, \mathcal{G} \subseteq 2^{[k]}$ be cross $t$-intersecting families and $1\le i< j\le k$. Suppose that $a$ is the maximal necessary intersection point of $\mathcal{F}$ and $ \mathcal{G}$. Then 
      $
       \left|[a-1] \cap F\cap G\right|\ge t-1.
      $
 Moreover, $s_{ij}(\mathcal{F})$ and $ s_{ij}(\mathcal{G})$ are  cross $t$-intersecting with the maximal necessary intersection point at most $a$. 
 \end{lemma}

\begin{proof}
  For the first assertion, assume that
      $
       \left|[a-1] \cap F\cap G\right|< t-1.
      $
  Then 
      $
        \left|[a] \cap F\cap G\right|\le \left|[a-1] \cap F\cap G\right|+1<t,
      $
which means there exists a necessary intersection point larger than $a$ and it is a contradiction.

For the second assertion, by Lemma \ref{G23}, it suffices to prove that 
 the maximal necessary intersection point of $s_{ij}(\mathcal{F})$ and $ s_{ij}(\mathcal{G})$  is  at most $a$.
Let $a'$ be their  maximal necessary intersection point.
Then there exist $F \in \mathcal{F}$ and $ G  \in \mathcal{G} $ such that
$ \left|[a'] \cap s_{ij}(F)\cap s_{ij}(G)\right|= t $ and $a'\in s_{ij}(F)\cap s_{ij}(G)$. 
If $a' >a$, then $\left|[a] \cap s_{ij}(F)\cap s_{ij}(G)\right|<t$. 
By the definition of $a$,  we have $s_{ij}(F)\notin\mathcal{F}$ or 
 $s_{ij}(G)\notin\mathcal{G}$.
By symmetry, we may assume that $s_{ij}(F)\notin\mathcal{F}$.
This implies that $ s_{ij}(F)=F\backslash\{j\}\cup \{i\}$.
From $\left|[a] \cap (F\backslash\{j\}\cup \{i\})\cap s_{ij}(G)\right|<t$
and $\left|[a] \cap F\cap G\right|\geq t$, we infer that 
$j\in F\cap G$ and $i\notin G$ and $G\backslash\{j\}\cup \{i\}\in \mathcal{G}$. Then $s_{ij}(G)=G$ and hence $\left|[a] \cap (F\backslash\{j\}\cup \{i\})\cap G\right|<t$. Furthermore, by the definition of $a$, we have
$$\left|[a]\cap F\cap (G\backslash\{j\}\cup \{i\}) \right|\geq t.$$
However, 
$|[a]\cap F\cap (G\backslash\{j\}\cup \{i\})|= \left|[a] \cap (F\backslash\{j\}\cup \{i\})\cap G\right|< t$, a contradiction. Therefore, we conclude that $a^{\prime}\leq a$.
\end{proof}

 Let $\mathcal{F}, \mathcal{G} \subseteq2^{[k]}$ be cross $t$-intersecting families with the maximal necessary intersection point $a$. 
Define
\begin{align*}
\mathcal{F}^a=&\left\{F\in \mathcal{F}:  \exists~ G  \in \mathcal{G} \text{ s.t. } \left|[a] \cap F\cap G\right|=t \text { and } a \in F\cap G\right\},\\
\mathcal{G}^a=&\left\{G\in \mathcal{G}:  \exists~ F  \in \mathcal{F} \text{ s.t. } \left|[a] \cap F\cap G\right|=t \text { and } a \in F\cap G\right\}.
\end{align*}
Additionally, let $\omega:[0, k] \rightarrow \mathbb{R}_{\geq0}$ be a \textit{non-increasing function}, meaning that $\omega(i) \geq \omega(j)$ whenever $i \leq j$. For any set $F \subseteq [n]$, define its weight as $\omega(F) = \omega(|F|)$. For a family $\mathcal{F} \subseteq 2^{[n]}$, define the total weight by $\omega(\mathcal{F}) = \sum_{F \in \mathcal{F}} \omega(F)$.
 
The following lemma, which is an analogue of Lemma 3.3 in \cite{G23}, plays a key role in our weighting argument.

\begin{lemma}\label{35}
     Let $k \geq \ell \geq \ell' \geq t\geq 1$  be integers. For $i \in[2]$, let $\omega_i:[0,k] \rightarrow \mathbb{R}_{\geq 0}$ be non-increasing functions.  Let $\mathcal{F}\subseteq\binom{[k]}{\leq \ell}$ and $\mathcal{G} \subseteq\binom{[k]}{\leq \ell'}$ be non-empty cross $t$-intersecting families with the maximal necessary intersection point $a$. Suppose that $a \geq t+1$ and $\mathcal{F} \backslash \mathcal{F}^a \neq \emptyset$ and $\mathcal{G} \backslash \mathcal{G}^a \neq \emptyset$. Then there are families $\mathcal{F}^*\subseteq\binom{[k]}{\leq \ell}$ and $\mathcal{G}^* \subseteq\binom{[k]}{\leq \ell'}$  such that
     \begin{itemize}
         \item [(1)] $\mathcal{F}^*, \mathcal{G}^*$ are non-empty cross $t$-intersecting families;
        \item[(2)]  $\omega_1(\mathcal{F}^*)+\omega_2(\mathcal{G}^*)\ge \omega_1(\mathcal{F})+\omega_2(\mathcal{G})$;
        \item[(3)] the maximal necessary intersection point of $\mathcal{F}^*$ and $\mathcal{G}^*$ is smaller than $a$.
         \item[(4)] there are shifted families $\mathcal{F}^*$ and $\mathcal{G}^*$ satisfying (1)-(3).
     \end{itemize}
\end{lemma}
\begin{proof}
Observe that (1) and (2) are preserved under shifting. According to Lemma \ref{34}, 
(3) is also preserved under shifting.
 Since one can apply the shifting operator to make $\mathcal{F}^*\subseteq\binom{[k]}{\leq \ell}$ and $\mathcal{G}^* \subseteq\binom{[k]}{\leq \ell'}$ shifted families, it suffices to establish properties (1)-(3).

Define
    $
    \mathcal{F}^{add}=\left\{F\backslash\{a\} : F\in \mathcal{F}^a   \right\},\ \mathcal{G}^{add}=\left\{G\backslash\{a\} : G\in \mathcal{G}^a   \right\}.
    $
Furthermore, let 
$$\mathcal{F}^+=\mathcal{F}\cup\mathcal{F}^{add},\quad \mathcal{G}^+=\mathcal{G}\cup\mathcal{G}^{add},\quad \mathcal{F}^-=\mathcal{F}\backslash\mathcal{F}^{a},\quad \mathcal{G}^-=\mathcal{G}\backslash\mathcal{G}^{a}.$$
Then  $\mathcal{F}^+, \mathcal{F}^-\subseteq \binom{[k]}{\leq \ell}$ and 
$\mathcal{G}^+, \mathcal{G}^-\subseteq \binom{[k]}{\leq \ell'}$.
\begin{claim}\label{cl1}
$\mathcal{F}\cap\mathcal{F}^{add}=\emptyset$ and $\mathcal{G}\cap\mathcal{G}^{add}=\emptyset$.
    \end{claim}
    \begin{proof}[Proof of Claim \ref{cl1}]
By symmetry, it suffices to prove that $\mathcal{F}\cap\mathcal{F}^{add}=\emptyset$.
        Assume, for the sake of contradiction, that  there exists $F\in\mathcal{F}\cap\mathcal{F}^{add}$.
     Since $F\in\mathcal{F}^{add}$, there is $F^*\in\mathcal{F}^a$
such that $F=F^*\backslash \{a\}$. By the definition of 
$\mathcal{F}^a$, there exists  $G\in \mathcal{G}$
 such that 
        $|[a]\cap F^*\cap G|=t$
        and $a\in F^*\cap G$.
This implies that $|[a-1]\cap F^*\cap G|=t-1$ and hence
        $
        |[a-1]\cap F\cap G|=t-1.
        $
Since $a\notin F$,  we obtain $|[a]\cap F\cap G|=t-1$, which contradicts with the maximality of $a$.
    \end{proof}

Since $ \omega_1, \omega_2$ are non-increasing, we obtain $\omega_1(\mathcal{F}^{add})\geq \omega_1(\mathcal{F}^a)$ and $\omega_2(\mathcal{G}^{add})\geq \omega_2(\mathcal{G}^a)$.
If $\omega_1(\mathcal{F}^a)\geq \omega_2(\mathcal{G}^a)$, then let
 $\mathcal{F}^*= \mathcal{F}^+$ and $\mathcal{G}^*=\mathcal{G}^-$.
It follows from Claim \ref{cl1} that
 \begin{align*}
       \omega_1(\mathcal{F}^*)+\omega_2(\mathcal{G}^*)&= \omega_1(\mathcal{F})+\omega_2(\mathcal{G})+\omega_1(\mathcal{F}^{add})-\omega_2(\mathcal{G}^a)\\
       &\geq\omega_1(\mathcal{F})+\omega_2(\mathcal{G})+\omega_1(\mathcal{F}^{a})-\omega_2(\mathcal{G}^a)\\
&\geq \omega_1(\mathcal{F})+\omega_2(\mathcal{G}).
        \end{align*}
If $\omega_1(\mathcal{F}^a)\leq \omega_2(\mathcal{G}^a)$, then let
 $\mathcal{F}^*= \mathcal{F}^-$ and $\mathcal{G}^*=\mathcal{G}^+$.
Similarly, we have $\omega_1(\mathcal{F}^*)+\omega_2(\mathcal{G}^*)\geq \omega_1(\mathcal{F})+\omega_2(\mathcal{G})$. Thus (2) holds.

It remains  to prove the following claim. 
 \begin{claim}\label{cl2}
  For any $F^+\in\mathcal{F}^+$ and $G^-\in\mathcal{G}^-$, 
     $
    |[a-1] \cap F^+\cap G^-|\geq t
     $
     holds.
Moreover, for any $F^-\in\mathcal{F}^-$ and $G^+\in\mathcal{G}^+$, 
     $
    |[a-1] \cap F^-\cap G^+|\geq t
     $
     holds.
    \end{claim}
    \begin{proof}[Proof of Claim \ref{cl2}]
By symmetry, it suffices to prove the first assertion.
    Suppose that there exist $F^+\in\mathcal{F}^+$ and $G^-\in\mathcal{G}^-$ such that
     $
    |[a-1] \cap F^+\cap G^-|< t.
     $

  If $F^+\in\mathcal{F}$,  then $G^-\in\mathcal{G}$ and
Lemma \ref{34} imply $
    |[a-1] \cap F^+\cap G^-|=t-1.
     $
By the maximality of $a$, we have $
    |[a] \cap  F^+\cap G^-|=t
     $
and $a\in F^+\cap G^-$.
Then $ G^-\in \mathcal{G}^a$,  a contradiction.

If $F^+\notin\mathcal{F}$, then $F^+\in\mathcal{F}^{add}$.  Since  $F^+\in\mathcal{F}^{add}$, there exists 
 $F\in\mathcal{F}^a$ such that
        $ F^+=F\backslash \{a\}$.
In view of    $
          |[a-1] \cap  F^+\cap G^-|< t,
        $
it follows that
      $
        \left|[a-1]\cap F\cap G^- \right|<t.
       $
Then $G^-\in\mathcal{G}$ and
Lemma \ref{34} imply 
$
        \left|[a-1]\cap F\cap G^- \right|=t-1.
       $
By the maximality of $a$, we have $
    |[a] \cap F\cap G^-|=t
     $
and $a\in  F\cap G^-$,
       which contradicts with  $ G^-\in \mathcal{G}\backslash\mathcal{G}^a$.
    \end{proof}

Since  Claim \ref{cl2} implies that (1) and (3) hold, the proof of the lemma is finished. 
\end{proof}

\begin{theorem}\label{n1}
Let $\mathcal{F}\subseteq\binom{[k]}{\leq \ell}$ and $\mathcal{G} \subseteq\binom{[k]}{\leq \ell'}$ be non-empty cross $t$-intersecting families with $k \geq \ell \geq \ell' \geq t\geq 1$.  For $i \in[2]$, let $\omega_i:[0,k] \rightarrow \mathbb{R}_{\geq 0}$ be non-increasing functions.   Then
\begin{align*}
\omega_1(\mathcal{F})+\omega_2(\mathcal{G}) \leq \max &\left\{\max _{a\in[t, \ell']}\{ \omega_1(\mathcal{K}_k(a, \ell, t))+\omega_2(\mathcal{S}_k(a, \ell'))\},\right.\\
&\left. \max _{a\in [t,\ell]}\{ \omega_1( \mathcal{S}_k(a,\ell))+\omega_2(\mathcal{K}_k(a, \ell', t))\}\right\},
\end{align*}
where $ \mathcal{K}_k(a, \ell, t) =\left\{K \in\binom{[k]}{\leq \ell}:|K \cap[a]| \geq t\right\}$ and $ \mathcal{S}_k(a,\ell')  =\left\{S \in\binom{[k]}{\leq \ell'}:[a] \subseteq S\right\}.$
\end{theorem}
\begin{proof}
 Let $\mathcal{F}\subseteq\binom{[k]}{\leq \ell}$ and $\mathcal{G} \subseteq\binom{[k]}{\leq \ell'}$ be  non-empty cross $t$-intersecting families with $\omega_1(\mathcal{F})+\omega_2(\mathcal{G})$  maximal and the maximal necessary intersection point minimal. 
  By Lemmas \ref{G23} and \ref{34}, we may assume that $\mathcal{F}$ and $\mathcal{G}$ are shifted. Let $a$ be  their  maximal necessary intersection point.

 If $a=t$,  then  for any $F\in\mathcal{F}$ and $G\in\mathcal{G}$, we have
  $|[t]\cap F\cap G|\ge t$, which means that $[t]\subseteq  F\cap G$. 
 Therefore, we have $\mathcal{F}=\mathcal{S}_k(a,\ell)$ and  $\mathcal{G}=\mathcal{S}_k(a,\ell')$, as required.

If $a\ge t+1$ and  $\mathcal{F} \backslash \mathcal{F}^a \neq \emptyset$ and $\mathcal{G} \backslash \mathcal{G}^a \neq \emptyset$, then Lemma \ref{35} implies that there exist  shifted and non-empty cross $t$-intersecting families
  $\mathcal{F}^*\subseteq\binom{[k]}{\leq \ell}$ and $\mathcal{G}^* \subseteq\binom{[k]}{\leq \ell'}$ such that  $\omega_1(\mathcal{F}^*)+\omega_2(\mathcal{G}^*)\ge \omega_1(\mathcal{F})+\omega_2(\mathcal{G})$. In addition, their maximal necessary intersection point  is smaller than $a$. This contradicts the minimality of $a$.
Therefore, either  $\mathcal{F}= \mathcal{F}^a $ or $\mathcal{G} = \mathcal{G}^a $. For the former case,  we have that all sets in $\mathcal{F}$ contain $a$. We claim that all sets in $\mathcal{F}$ contain $[a]$. Indeed, $F\in\mathcal{F}$ and $b\in[a]$ such that $b\notin F$ would imply $F\setminus \{a\}\cup\{b\}\in \mathcal{F}$ because $\mathcal{F}$ is shifted, a contradiction.
Therefore, we have $\mathcal{F}=\mathcal{S}_k(a,\ell)$ and hence $\mathcal{G}=\mathcal{K}_k(a, \ell', t)$.
For the latter case,  a similar argument yields  
$\mathcal{F}=\mathcal{K}_k(a, \ell, t)$ and $\mathcal{G}=\mathcal{S}_k(a,\ell')$. This completes the proof.
\end{proof}

\subsection{Proof of Theorem \ref{t3}}
Having established Lemma \ref{4} and Theorem \ref{n1}, we now return  to the proof of Theorem \ref{t3}.

 \begin{proof}[Proof of Theorem \ref{t3}]
     Let $\mathcal{F}\subseteq\mathcal{H}(n,k,\ell)$  and  $\mathcal{G}\subseteq\mathcal{H}(n,k,\ell')$ be non-empty cross $t$-intersecting families with $|\mathcal{F}|+|\mathcal{G}| $ maximal. 
     By Lemma \ref{G23}, we may assume that they are shifted 
separated families. 
According to Lemma \ref{4}, $\mathcal{A}(\mathcal{F})\subseteq \binom{[k]}{\le \ell}$ and $\mathcal{A}(\mathcal{G})\subseteq \binom{[k]}{\le \ell'}$ are cross $t$-intersecting. 

Observe that
     \begin{align*}
     |\mathcal{F}|\leq\sum\limits_{j\in[0,\ell]} 
     \left|\mathcal{A}(\mathcal{F})^{(j)}\right|\binom{k-j}{\ell-j}(n-1)^{\ell-j},\quad
|\mathcal{G}|\leq\sum\limits_{j\in[0,\ell']} 
     \left|\mathcal{A}(\mathcal{G})^{(j)}\right|\binom{k-j}{\ell'-j}(n-1)^{\ell'-j}. 
     \end{align*}
Let us define
     $$
     \omega_1{(j)}=\binom{k-j}{\ell-j}(n-1)^{\ell-j},\quad  \omega_2{(j)}=\binom{k-j}{\ell'-j}(n-1)^{\ell'-j},\quad j\in [0,k].
     $$
Since $n\geq 2$ and $k\geq \ell\geq \ell'$, we infer that $ \omega_1$ and $ \omega_2$
are non-increasing functions. 

Recall that  $ \mathcal{K}_k(a, \ell, t) =\left\{K \in\binom{[k]}{\leq \ell}:|K \cap[a]| \geq t\right\}$ and $ \mathcal{S}_k(a,\ell')  =\left\{S \in\binom{[k]}{\leq \ell'}:[a] \subseteq S\right\}.$
It follows from Theorem \ref{n1} that
     \begin{align*}
        |\mathcal{F}|+ |\mathcal{G}|&\le  \sum\limits_{j\in[0,\ell]} 
     \left|\mathcal{A}(\mathcal{F})^{(j)}\right|\binom{k-j}{\ell-j}(n-1)^{\ell-j}+\sum\limits_{j\in[0,\ell']} 
     \left|\mathcal{A}(\mathcal{G})^{(j)}\right|\binom{k-j}{\ell'-j}(n-1)^{\ell'-j}\\
&= \sum\limits_{j\in[0,\ell]}|\mathcal{A}(\mathcal{F})^{(j)}| \omega_1(j)+\sum\limits_{j\in[0,\ell']}|\mathcal{A}(\mathcal{G})^{(j)}| \omega_2(j)\\
&=\omega_1(\mathcal{A}(\mathcal{F}))+\omega_2(\mathcal{A}(\mathcal{G}))\\
       &\le\max\left\{\max _{a\in[t, \ell']}\{ \omega_1(\mathcal{K}_k(a, \ell, t))+\omega_2(\mathcal{S}_k(a, \ell'))\},\ \max _{a\in [t,\ell]}\{ \omega_1( \mathcal{S}_k(a,\ell))+\omega_2(\mathcal{K}_k(a, \ell', t))\}\right\}.
     \end{align*}
Let $f(a)=\omega_1(\mathcal{K}_k(a, \ell, t))+\omega_2(\mathcal{S}_k(a, \ell'))$ and 
$g(a)=\omega_1( \mathcal{S}_k(a,\ell))+\omega_2(\mathcal{K}_k(a, \ell', t))$.
Since
 \begin{align*}
\mathcal{K}_k(a+1, \ell, t)\backslash \mathcal{K}_k(a, \ell, t)&=\left\{K \in\binom{[k]}{\leq \ell}:|K \cap[a]|=t-1\text{ and } a+1\in K\right\},\\
\mathcal{S}_k(a, \ell')\backslash\mathcal{S}_k(a+1, \ell')&=\left\{K \in\binom{[k]}{\leq \ell'}: [a]\subseteq K \text{ and } a+1\notin K\right\},
 \end{align*}
it  follows that for any $a\in [t,\ell'-1]$, we have
  \begin{align*}
f(a+1)-f(a)=&\sum\limits_{j\in[0,\ell-t]} 
     \binom{a}{t-1}\binom{k-a-1}{j}\omega_1{(j+t)}-\sum\limits_{j\in[0,\ell'-a]} 
     \binom{k-a-1}{j}\omega_2{(j+a)}.
  \end{align*}
From $\ell\geq \ell'$ and $a\geq t$, we get $\omega_1{(j+t)}\geq \omega_2{(j+t)}\geq \omega_2{(j+a)}$ and $\ell-t\geq \ell'-a$.
This implies that
$f(a+1)-f(a)\geq 0$ and hence $f(\ell')=\max\limits _{a\in[t, \ell']} f(a)$.

 For any $a\in [t,\ell']$, we have $f(a)\geq g(a)$ because $
\mathcal{K}_k(a, \ell', t)\backslash \mathcal{S}_k(a, \ell')\subseteq \mathcal{K}_k(a, \ell, t)\backslash \mathcal{S}_k(a, \ell)
$.
This, together with $ f(\ell')=\max\limits _{a\in[t, \ell']} f(a)$, implies that
 \begin{align*}
        |\mathcal{F}|+ |\mathcal{G}|\le\left\{f(\ell'),\ \max _{a\in [\ell'+1,\ell]}g(a)\right\}= \max \left\{|\mathcal{F}_0|+|\mathcal{G}_0|,\ \max_{a\in[\ell'+1,\ell]}\left\{|\mathcal{F}_a|+|\mathcal{G}_a|\right\}\right\},
     \end{align*}
as required.
 \end{proof}

%\section*{Declaration of competing interest}
%We declare that we have no conflict of interest to this work.

%\section*{Data availability}
%No data was used for the research described in the article.

\section*{Acknowledgement}
The authors express their sincere thanks to the referee for the valuable suggestions which greatly improved the presentation of the manuscript. The corresponding author Lihua Feng was supported by the NSFC (Nos. 12271527 and 12471022) and NSF of Qinghai Province (No. 2025-ZJ-902T), Yongtao Li as the  corresponding author is a postdoctor at Tsinghua University.

\end{document}